\documentclass[runningheads]{llncs}
\usepackage[T1]{fontenc}
\usepackage{graphicx}
\usepackage{amsmath,amsfonts,amssymb}
\usepackage{tikz}
\usetikzlibrary{patterns}
\usepackage{bm}
\usepackage{comment}
\usepackage{xcolor}

\begin{document}
\title{On 1-skeleton of the cut polytopes}
%
%
\author{Andrei V. Nikolaev\orcidID{0000-0003-4705-2409}}
\authorrunning{A.V. Nikolaev}
%
\institute{P.G. Demidov Yaroslavl State University, Yaroslavl, Russia
	\email{andrei.v.nikolaev@gmail.com}}
\maketitle              
\begin{abstract}
Given an undirected graph $G = (V,E)$, the cut polytope $\mathrm{CUT}(G)$ is defined as the convex hull of the incidence vectors of all cuts in $G$.	The 1-skeleton of $\mathrm{CUT}(G)$ is a graph whose vertex set is the vertex set of the polytope, and the edge set is the set of geometric edges or one-dimensional faces of the polytope. 

We study the diameter and the clique number of 1-skeleton of cut polytopes for several classes of graphs. These characteristics are of interest since they estimate the computational complexity of the max-cut problem for certain computational models and classes of algorithms.

It is established that while the diameter of the 1-skeleton of a cut polytope does not exceed $|V|-1$ for any connected graph, the clique number varies significantly depending on the class of graphs. For trees, cacti, and almost trees (2), the clique number is linear in the dimension, whereas for complete bipartite and $k$-partite graphs, it is superpolynomial.


\end{abstract}

\section{Introduction}

The 1-\textit{skeleton} of a polytope $P$ is a graph whose vertex set is the vertex set of $P$ and edge set is the set of geometric edges or one-dimensional faces of $P$.
We are interested in 1-skeletons of polytopes arising from linear and integer programming models since some of their characteristics help to estimate the time complexity for different computation models and classes of algorithms.

The \textit{diameter} of a graph $G$, denoted by $d(G)$, is the maximum edge distance between any pair of vertices.
The diameter of 1-skeleton is known as a lower bound on the number of iterations of the simplex method and similar algorithms.
Indeed, let the shortest path between a pair of vertices $\mathbf{u}$ and $\mathbf{v}$ of a polytope $P$ consist of $d(P)$ edges. If the simplex method chooses $\mathbf{u}$ as the initial solution, and the optimal solution is $\mathbf{v}$, then no matter how successfully the algorithm chooses the next adjacent vertex of the 1-skeleton, the number of iterations cannot be less than $d(P)$ (see, for example, Dantzig~\cite{Dantzig1991}).

The \textit{clique number} of a graph $G$, denoted by $\omega(G)$, is the number of vertices in the largest complete subgraph.
The clique number of 1-skeleton serves as a lower bound for computational complexity in a class of \textit{direct-type} algorithms based on linear decision trees.
The idea is that direct-type algorithms cannot, in the worst case, discard more than one of the pairwise adjacent solutions in one iteration.
This class includes, for example, some classical dynamic programming, branch-and-bound and other algorithms. See Bondarenko~\cite{Bondarenko1983,Bondarenko1993} for more details.
Besides, for all known cases, the clique number of 1-skeleton of a polytope is polynomial for polynomially solvable problems (Bondarenko and Nikolaev~\cite{Bondarenko2016,Bondarenko2018,Nikolaev2022}) and superpolynomial for intractable problems (Bondarenko et al.~\cite{Bondarenko1983,Bondarenko2017SPT,Bondarenko2017Bi}, Padberg~\cite{Padberg1989}, and Simanchev \cite{Simanchev2018}).

In this paper, we consider 1-skeletons of cut polytopes associated with the maximum cut problem for several classes of graphs, in particular for trees, cacti, almost trees (2), complete bipartite, and $k$-partite graphs.
The results of the research are summarized in Table~\ref{Table_pivot_1-skeleton_cut_polytopes}. New results are highlighted in bold.

\begin{table}[t]
	\centering
	\caption{Summary table of properties of the 1-skeleton of the cut polytope $\mathrm{CUT}(G)$ for some classes of graphs}
	\label{Table_pivot_1-skeleton_cut_polytopes}
		\begin{tabular}{c|c|c|c}
					& Max-cut & Graph & Clique \\
					& complexity & diameter & number \\ \hline
			\begin{tabular}{@{}c@{}} \\ ~Tree~ \\ ~\end{tabular}	& P~\cite{Hadlock1975,Orlova1972} & $|V|-1$~\cite{Neto2016} &  2~\cite{Neto2016} \\ \hline
			
			\begin{tabular}{@{}c@{}} \\ ~Cactus~ \\ ~\end{tabular}	& P~\cite{Hadlock1975,Orlova1972} & ~$\bm{\lfloor \frac{|V|}{2} \rfloor\leq d \leq |V|-1}$~ & ~$\bm{\leq 2^{\lceil\log_2\left(|E|+1\right)\rceil} \leq 2|E|}$~ \\ \hline
			
			\begin{tabular}{@{}c@{}} Almost~ \\ tree (2) \end{tabular} & P~\cite{Hadlock1975,Orlova1972} & $\bm{\lfloor \frac{|V|}{3} \rfloor\leq d \leq |V|-1}$ & $\bm{\leq 2^{\lceil\log_2\left(|E|\right)\rceil + 1} \leq 4|E|}$ \\ \hline
			
			\begin{tabular}{@{}c@{}} ~~Complete~~ \\ ~bipartite~ \\ graph \end{tabular}	& ~NP-hard~\cite{McCormick2003}~ & 2~\cite{Neto2016} & $\bm{\geq 2^{\min\{|V_1|,|V_2|\}-1}}$ \\ \hline
			
			\begin{tabular}{@{}c@{}} ~Complete~ \\ $k$-partite \\ graph\end{tabular} & NP-hard~\cite{Bodlaender1994} & ~$\bm{2}$ & $\bm{\geq 2^{\textbf{2nd largest}\{|V_1|,\ldots,|V_k|\} - 1}}$~ \\ \hline
			
			\begin{tabular}{@{}c@{}} Complete \\ graph \end{tabular} & NP-hard~\cite{Karp1972} & $1$~\cite{Barahona1986} & $2^{|V|-1}$~\cite{Barahona1986}
		\end{tabular}
\end{table}

\section{Cut polytopes}
We consider the classical max-cut problem in an undirected weighted graph.

\vspace{2mm}

\textbf{\textsc{Maximum cut problem (max-cut).}}

\textsc{Instance.} Given an undirected graph $G=(V,E)$ with an edge weight function $w:E \rightarrow \mathbb{Z}$.

\textsc{Question.} Find a subset of vertices $S \subseteq V$ such that the sum of the weights of the edges from $E$ with one endpoint in $S$ and another in $V \backslash S$ (\textit{cut}) is as large as possible.

\vspace{2mm}

Max-cut is a well-known NP-hard problem (Karp~\cite{Karp1972}, Garey and Johnson~\cite{Garey1979}) that arises, for example, in cluster analysis (Boros and Hammer~\cite{Boros1989}), the Ising model in statistical physics (Barahona et al.~\cite{Barahona1988}), the VLSI design (Chen at al.~\cite{Chen1983}), social network analysis (Agrawal et al.~\cite{Agrawal2003}), and the image segmentation (de Sousa at al.~\cite{deSousa2013}).

The cut polytopes were introduced by Barahona and Mahjoub in their seminal paper~\cite{Barahona1986}. 
Let $G=(V,E)$ be an undirected graph with node set $V$ and edge set $E$.
With each subset $S \subseteq V$ we associate the incidence $0/1-$vector $\mathbf {v}(S) \in \{0,1\}^{|E|}$:
\[
v(S)_{e} =
\begin{cases}
	1, &\text{if } e \in \delta(S),\\
	0, &\text{otherwise}.
\end{cases},
\]
where $\delta(S) \subseteq E$ is the \textit{cut-set} of edges with exactly one endpoint in $S$.

The \textit{cut polytope} $\mathrm{CUT}(G)$ is defined as the convex hull of all incidence vectors:
\[\mathrm{CUT}(G) = \operatorname{conv} {\{\mathbf{v} (S): S \subseteq V\}} \subset \mathbb{R}^{|E|}.\]

An example of constructing the cut polytope $\mathrm{CUT}(K_3)$ for the complete graph $K_3$ on $3$ vertices is shown in Fig.~\ref{Fig_CUT(K_3)}.

\begin{figure}[t]
	\centering
	\begin{tikzpicture}[scale=0.95]
		\node[circle,draw,inner sep=3pt,fill=blue!35] (A) at (-30:1) {};
		\node[circle,draw,inner sep=3pt,fill=blue!35] (B) at (90:1) {};
		\node[circle,draw,inner sep=3pt,fill=blue!35] (C) at (210:1) {};
		
		\draw (A) -- (B) -- (C) -- (A);
		
		\begin{scope}[xshift=3cm]
			\node[circle,draw,inner sep=3pt,fill=red!35] (A) at (-30:1) {};
			\node[circle,draw,inner sep=3pt,fill=blue!35] (B) at (90:1) {};
			\node[circle,draw,inner sep=3pt,fill=blue!35] (C) at (210:1) {};
			
			\draw (A) -- (B) -- (C) -- (A);
			
			\draw [thick,red] (A) -- (B);
			\draw [thick,red] (A) -- (C);
			
			\draw [purple,dashed,thick] (45:1.2) -- (255:1.2);
		\end{scope}

		\begin{scope}[xshift=6cm]
			\node[circle,draw,inner sep=3pt,fill=blue!35] (A) at (-30:1) {};
			\node[circle,draw,inner sep=3pt,fill=red!35] (B) at (90:1) {};
			\node[circle,draw,inner sep=3pt,fill=blue!35] (C) at (210:1) {};
			
			\draw (A) -- (B) -- (C) -- (A);
			
			\draw [thick,red] (B) -- (A);
			\draw [thick,red] (B) -- (C);
			
			\draw [purple,dashed,thick] (15:1.2) -- (165:1.2);
		\end{scope}

		\begin{scope}[xshift=9cm]
			\node[circle,draw,inner sep=3pt,fill=blue!35] (A) at (-30:1) {};
			\node[circle,draw,inner sep=3pt,fill=blue!35] (B) at (90:1) {};
			\node[circle,draw,inner sep=3pt,fill=red!35] (C) at (210:1) {};
			
			\draw (A) -- (B) -- (C) -- (A);
			
			\draw [thick,red] (C) -- (A);
			\draw [thick,red] (C) -- (B);
			
			\draw [purple,dashed,thick] (135:1.2) -- (285:1.2);
		\end{scope}
		
		\node at (4.5,-1.5) {$\mathrm{CUT}(K_3) = \operatorname{conv}\left\lbrace (0,0,0), (0,1,1), (1,0,1), (1,1,0) \right\rbrace$};
		
	\end{tikzpicture}
	\caption {An example of constructing a cut polytope for $K_3$}
	\label {Fig_CUT(K_3)}
\end{figure}
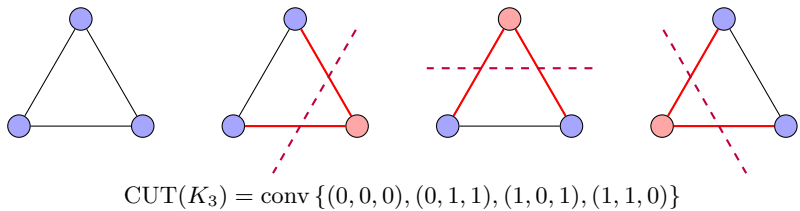

Note that since $\delta(S) = \delta (V \backslash S)$, then any graph $G = (V,E)$ contains exactly $2^{|V|-1}$ cuts, and the cut polytope $\mathrm{CUT}(G)$ will have the same number of vertices.

The cut polytope and its various relaxations often serve as the linear programming models for the maximum cut problem. See, for example, the polynomial time algorithm by Barahona for the max-cut on graphs not contractible to $K_5$~\cite{Barahona1983}. For more details, see the monograph by Deza and Laurent~\cite{DezaLaurent1997}.

\section{Adjacency}

The 1-skeletons of cut polytopes were first studied by Barahona and Mahjoub~\cite{Barahona1986}. In particular, they described a criterion for the adjacency of vertices in a cut polytope.

We denote by $X \triangle Y$ the symmetric difference of the sets $X$ and $Y$. 
It is easy to check that for any cuts $X,Y \subseteq V$ we have $\delta(X) \triangle \delta (Y) = \delta(X \triangle Y)$.

\begin{theorem}[Barahona and Mahjoub~\cite{Barahona1986}] \label{Theorem_Barahona_adjacency}
	Let $G = (V,E)$ be a connected graph. Let $\mathbf{v}(X)$ and $\mathbf{v}(Y)$ be extreme points of $\mathrm{CUT}(G)$, where $X,Y \subseteq V$ are the corresponding cuts. Then $\mathbf{v}(X)$ and $\mathbf{v}(Y)$ are adjacent in $\mathrm{CUT}(G)$ if and only if the graph $H_{X \triangle Y} = (V, E \backslash (\delta(X \triangle Y))$ has two connected components.
\end{theorem}

A short summary of the proof is as follows.
Let $G_i = (V_i, E_i)$, $i = 1,2$ be two connected components of
$H_{X \triangle Y} = (V, E \backslash (\delta(X \triangle Y))$, and let $T_i$ be a spanning tree of $G_i$, $i = 1, 2$. Then we can consider vector $\mathbf{c} \in \mathbb{R}^{|E|}$:
\[
c_{e} = \begin{cases}
	1, &\text{if } e \in T_{i} \cap (\delta(X) \cap \delta(Y)),\ i=1,2,\\
	-1, &\text{if } e \in T_{i} \backslash (\delta(X) \cap \delta(Y)),\ i=1,2,\\
	0, &\text{otherwise},
\end{cases}
\]
such that $\mathbf{v}(X)$ and $\mathbf{v}(Y)$ are the only two extreme points
that maximize $\mathbf{c}^{T} \mathbf{v}$ over $\mathrm{CUT}(G)$ (see Fig.~\ref{Fig_example_adjacent_cuts}), therefore they are adjacent in the 1-skeleton.

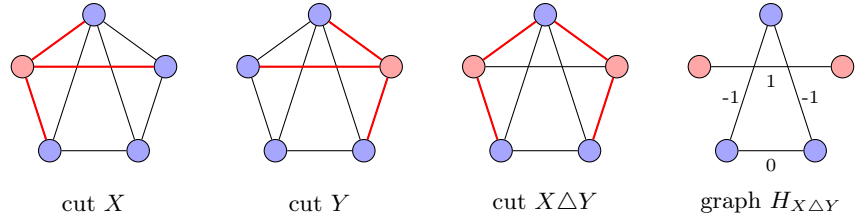
\begin{figure}[t]
	\centering
	\begin{tikzpicture}[scale=1.0 ]
		\node[circle,draw,inner sep=3pt,fill=blue!35] (A) at (18:1) {};
		\node[circle,draw,inner sep=3pt,fill=blue!35] (B) at (90:1) {};
		\node[circle,draw,inner sep=3pt,fill=red!35] (C) at (162:1) {};
		\node[circle,draw,inner sep=3pt,fill=blue!35] (D) at (234:1) {};
		\node[circle,draw,inner sep=3pt,fill=blue!35] (E) at (306:1) {};
		
		\draw (A) -- (B) -- (C) -- (D) -- (E) -- (A);
		\draw (A) -- (C);
		\draw (B) -- (D);
		\draw (B) -- (E);
		
		\draw [thick,red] (C) -- (A);
		\draw [thick,red] (C) -- (B);
		\draw [thick,red] (C) -- (D);
		
		\node at (270:1.5) {cut $X$};
		

		\begin{scope}[xshift=3cm]
			\node[circle,draw,inner sep=3pt,fill=red!35] (A) at (18:1) {};
			\node[circle,draw,inner sep=3pt,fill=blue!35] (B) at (90:1) {};
			\node[circle,draw,inner sep=3pt,fill=blue!35] (C) at (162:1) {};
			\node[circle,draw,inner sep=3pt,fill=blue!35] (D) at (234:1) {};
			\node[circle,draw,inner sep=3pt,fill=blue!35] (E) at (306:1) {};
			
			\draw (A) -- (B) -- (C) -- (D) -- (E) -- (A);
			\draw (A) -- (C);
			\draw (B) -- (D);
			\draw (B) -- (E);
			
			\draw [thick,red] (A) -- (B);
			\draw [thick,red] (A) -- (C);
			\draw [thick,red] (A) -- (E);
			
			\node at (270:1.5) {cut $Y$};
		\end{scope}

		\begin{scope}[xshift=6cm]
			\node[circle,draw,inner sep=3pt,fill=red!35] (A) at (18:1) {};
			\node[circle,draw,inner sep=3pt,fill=blue!35] (B) at (90:1) {};
			\node[circle,draw,inner sep=3pt,fill=red!35] (C) at (162:1) {};
			\node[circle,draw,inner sep=3pt,fill=blue!35] (D) at (234:1) {};
			\node[circle,draw,inner sep=3pt,fill=blue!35] (E) at (306:1) {};
			
			\draw (A) -- (B) -- (C) -- (D) -- (E) -- (A);
			\draw (A) -- (C);
			\draw (B) -- (D);
			\draw (B) -- (E);
			
			\draw [thick,red] (A) -- (B);
			\draw [thick,red] (A) -- (E);
			\draw [thick,red] (C) -- (B);
			\draw [thick,red] (C) -- (D);
			
			\node at (270:1.5) {cut $X \triangle Y$};
		\end{scope}

		\begin{scope}[xshift=9cm]
			\node[circle,draw,inner sep=3pt,fill=red!35] (A) at (18:1) {};
			\node[circle,draw,inner sep=3pt,fill=blue!35] (B) at (90:1) {};
			\node[circle,draw,inner sep=3pt,fill=red!35] (C) at (162:1) {};
			\node[circle,draw,inner sep=3pt,fill=blue!35] (D) at (234:1) {};
			\node[circle,draw,inner sep=3pt,fill=blue!35] (E) at (306:1) {};
			
			\draw (A) -- node[below]{\scriptsize 1} (C);
			\draw (B) -- node[below left]{\scriptsize -1} (D);
			\draw (B) -- node[below right]{\scriptsize -1} (E);
			\draw (D) -- node[below]{\scriptsize 0} (E);
			
			\node at (270:1.5) {graph $H_{X \triangle Y}$};
		
		\end{scope}

	\end{tikzpicture}
	\caption {An example of adjacent cuts according to Theorem~\ref{Theorem_Barahona_adjacency}}
	\label {Fig_example_adjacent_cuts}
\end{figure}

Let now the graph $H_{X \triangle Y}$ contain more than two connected components (see Fig.~\ref{Fig_example_non_adjacent_cuts}), then at least one of these components defines a non-empty cut $L$, such that $\delta(L) \subset \delta(X \triangle Y)$, and
\[\mathbf{v}(X \triangle L) + \mathbf{v}(Y \triangle L) = \mathbf{v}(X) + \mathbf{v}(Y).\]
A line segment connecting vertices $\mathbf{v}(X)$ and $\mathbf{v}(Y)$ intersects a line segment connecting two other vertices, therefore, $\mathbf{v}(X)$ and $\mathbf{v}(Y)$ cannot be adjacent.

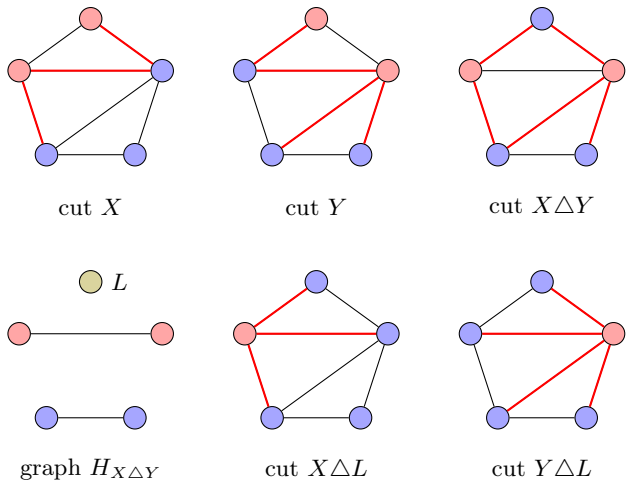
\begin{figure}[t]
	\centering
	\begin{tikzpicture}[scale=1.0 ]
		\node[circle,draw,inner sep=3pt,fill=blue!35] (A) at (18:1) {};
		\node[circle,draw,inner sep=3pt,fill=red!35] (B) at (90:1) {};
		\node[circle,draw,inner sep=3pt,fill=red!35] (C) at (162:1) {};
		\node[circle,draw,inner sep=3pt,fill=blue!35] (D) at (234:1) {};
		\node[circle,draw,inner sep=3pt,fill=blue!35] (E) at (306:1) {};
		
		\draw (A) -- (B) -- (C) -- (D) -- (E) -- (A);
		\draw (A) -- (C);
		\draw (A) -- (D);
		
		\draw [thick,red] (B) -- (A);
		\draw [thick,red] (C) -- (A);
		\draw [thick,red] (C) -- (D);
		
		\node at (270:1.5) {cut $X$};
		

		\begin{scope}[xshift=3cm]
			\node[circle,draw,inner sep=3pt,fill=red!35] (A) at (18:1) {};
			\node[circle,draw,inner sep=3pt,fill=red!35] (B) at (90:1) {};
			\node[circle,draw,inner sep=3pt,fill=blue!35] (C) at (162:1) {};
			\node[circle,draw,inner sep=3pt,fill=blue!35] (D) at (234:1) {};
			\node[circle,draw,inner sep=3pt,fill=blue!35] (E) at (306:1) {};
			
			\draw (A) -- (B) -- (C) -- (D) -- (E) -- (A);
			\draw (A) -- (C);
			\draw (A) -- (D);
			
			\draw [thick,red] (A) -- (C);
			\draw [thick,red] (A) -- (D);
			\draw [thick,red] (A) -- (E);
			\draw [thick,red] (B) -- (C);
			
			\node at (270:1.5) {cut $Y$};
		\end{scope}

		\begin{scope}[xshift=6cm]
			\node[circle,draw,inner sep=3pt,fill=red!35] (A) at (18:1) {};
			\node[circle,draw,inner sep=3pt,fill=blue!35] (B) at (90:1) {};
			\node[circle,draw,inner sep=3pt,fill=red!35] (C) at (162:1) {};
			\node[circle,draw,inner sep=3pt,fill=blue!35] (D) at (234:1) {};
			\node[circle,draw,inner sep=3pt,fill=blue!35] (E) at (306:1) {};
			
			\draw (A) -- (B) -- (C) -- (D) -- (E) -- (A);
			\draw (A) -- (C);
			\draw (A) -- (D);
			
			\draw [thick,red] (A) -- (B);
			\draw [thick,red] (A) -- (E);
			\draw [thick,red] (A) -- (D);
			\draw [thick,red] (C) -- (B);
			\draw [thick,red] (C) -- (D);
			
			\node at (270:1.5) {cut $X \triangle Y$};
		\end{scope}

		\begin{scope}[yshift=-3.5cm]
			\node[circle,draw,inner sep=3pt,fill=red!35] (A) at (18:1) {};
			\node[circle,draw,inner sep=3pt,fill=olive!35,label=right:{\small $L$}] (B) at (90:1) {};
			\node[circle,draw,inner sep=3pt,fill=red!35] (C) at (162:1) {};
			\node[circle,draw,inner sep=3pt,fill=blue!35] (D) at (234:1) {};
			\node[circle,draw,inner sep=3pt,fill=blue!35] (E) at (306:1) {};
			
			\draw (A) -- (C);
			\draw (D) -- (E);
			
			\node at (270:1.5) {graph $H_{X \triangle Y}$};
			
		\end{scope}
		
		\begin{scope}[yshift=-3.5cm, xshift=3cm]
			\node[circle,draw,inner sep=3pt,fill=blue!35] (A) at (18:1) {};
			\node[circle,draw,inner sep=3pt,fill=blue!35] (B) at (90:1) {};
			\node[circle,draw,inner sep=3pt,fill=red!35] (C) at (162:1) {};
			\node[circle,draw,inner sep=3pt,fill=blue!35] (D) at (234:1) {};
			\node[circle,draw,inner sep=3pt,fill=blue!35] (E) at (306:1) {};
			
			\draw (A) -- (B) -- (C) -- (D) -- (E) -- (A);
			\draw (A) -- (C);
			\draw (A) -- (D);
			
			\draw [thick,red] (C) -- (A);
			\draw [thick,red] (C) -- (B);
			\draw [thick,red] (C) -- (D);
			
			\node at (270:1.5) {cut $X \triangle L$};
			
		\end{scope}
		
		\begin{scope}[yshift=-3.5cm, xshift=6cm]
			\node[circle,draw,inner sep=3pt,fill=red!35] (A) at (18:1) {};
			\node[circle,draw,inner sep=3pt,fill=blue!35] (B) at (90:1) {};
			\node[circle,draw,inner sep=3pt,fill=blue!35] (C) at (162:1) {};
			\node[circle,draw,inner sep=3pt,fill=blue!35] (D) at (234:1) {};
			\node[circle,draw,inner sep=3pt,fill=blue!35] (E) at (306:1) {};
			
			\draw (A) -- (B) -- (C) -- (D) -- (E) -- (A);
			\draw (A) -- (C);
			\draw (A) -- (D);
			
			\draw [thick,red] (A) -- (B);
			\draw [thick,red] (A) -- (C);
			\draw [thick,red] (A) -- (D);
			\draw [thick,red] (A) -- (E);
			
			\node at (270:1.5) {cut $Y \triangle L$};
			
		\end{scope}
	\end{tikzpicture}
	\caption {An example of non-adjacent cuts according to Theorem~\ref{Theorem_Barahona_adjacency}}
	\label {Fig_example_non_adjacent_cuts}
\end{figure}

Theorem~\ref{Theorem_Barahona_adjacency} directly implies the following statement.
\begin{corollary}[Barahona and Mahjoub~\cite{Barahona1986}]
	If $G$ is a complete graph, then $\mathrm{CUT}(G)$ has diameter one.
\end{corollary}


The cut polytope of a complete graph is a classic example of a combinatorial polytope, any two vertices of which are pairwise adjacent. Note also that the clique number of the 1-skeleton is equal to $2^{|V|-1}$.

However, the situation may be fundamentally different for other classes of graphs. In this paper, we follow Neto~\cite{Neto2016}, who described the diameters of the 1-skeletons of cut polytopes for trees and complete bipartite graphs.

The study of cut polytopes for different classes of graphs is also of interest since the adjacency of vertices in the 1-skeleton is inherited from the subgraphs.

\begin{theorem}[Neto~\cite{Neto2016}] \label{Theorem_Neto_Subgraph}
	Let $G' = (V, E')$ be a connected subgraph of $G = (V, E)$, and let $X$, $Y$ denote two distinct node subsets of $V$. Then, if the vertices $\mathbf{v}(X)$ and $\mathbf{v}(Y)$ are adjacent in $\mathrm{CUT}(G')$, they are also adjacent in $\mathrm{CUT}(G)$.
\end{theorem}

Indeed, if the graph $H'_{X \triangle Y}$ for $G'$ has only 2 connected components, then these same components will remain connected in the graph $H_{X \triangle Y}$ for $G$.
Thus, the 1-skeleton of $\mathrm{CUT}(G')$ is a subgraph of the 1-skeleton of $\mathrm{CUT}(G)$.

\section{Trees, cacti, and almost trees}

In this section, we study cut polytopes for connected graphs with a small number of edges. In particular, trees, cacti, and almost trees (2).
Let us recall the definitions of the corresponding classes of graphs, following Brandst\"{a}dt et al.~\cite{Brandstadt1999}:
\begin{itemize}
	\item a \textit{tree} is a connected graph without cycles;
	\item a \textit{cactus} is a connected graph in which every edge belongs to at most one simple cycle;
	\item a connected graph is an \textit{almost tree} ($k$) if every \textit{biconnected component} (maximal subgraph that remains connected when any vertex is removed) has the property that there are at most $k$ edges not in a spanning tree of this biconnected component.
\end{itemize}

It's easy to see that $G$ is a cactus if $G$ is an almost tree (1). 
Some problems that are NP-hard for general graphs are polynomially solvable on almost trees with small parameter $k$ (see Ben-Moshe et al.~\cite{Ben-Moshe2005}, Coppersmith and Vishkin~\cite{Coppersmith1985}, Gurevich et al.~\cite{Gurevich1984}).
As for the max-cut problem, trees, cacti, and almost trees (2) are planar graphs, and it is well known that the maximum cut in a planar graph can be found in polynomial time (Orlova and Dorfman~\cite{Orlova1972}, Hadlock~\cite{Hadlock1975}), and, in general, on graphs not contractible to $K_5$ (Barahona~\cite{Barahona1983}).

\subsection{Trees}

Cut polytopes for trees were first considered by Neto~\cite{Neto2016} as an example of graphs with the largest possible value of the diameter of the corresponding cut polytopes. Here we provide an alternative proof of this result.

The \textit{Hamming distance} between two equal-length strings of symbols $\mathbf{x}$ and $\mathbf{y}$, denoted by $\Delta(\mathbf{x},\mathbf{y})$, is defined as the number of positions at which the entries of $\mathbf{x}$ and $\mathbf{y}$ are different.

\begin{theorem}[Neto~\cite{Neto2016}]
	If $G=(V,E)$ is a tree, then the cut polytope $\mathrm{CUT}(G)$ is a unit hypercube $\left[ 0, 1\right]^{|V|-1}$.
\end{theorem}
	
\begin{proof}
	Note that if we remove $k$ edges from the tree, then it splits into $k-1$ connected components. Therefore, for some two distinct cuts $X$ and $Y$ the graph $H_{X \triangle Y}$ has 2 connected components if and only if $|\delta(X \triangle Y)| = 1$. Thus, the incidence vectors of cuts in a tree are $0/1$ vectors of length $|E|=|V|-1$, and two vectors are adjacent if and only if the Hamming distance between them is exactly $1$. Clearly, this is a unit hypercube $\left[0, 1\right]^{|V|-1}$.	
\end{proof}

\begin{corollary}
	If $G = (V,E)$ is a tree, then the diameter $d(\mathrm{CUT}(G)) = |V|-1$, and the clique number  $\omega(\mathrm{CUT}(G)) = 2$.
\end{corollary}

Thus, taking into account Theorem~\ref{Theorem_Neto_Subgraph}, cut polytopes for trees, among all connected graphs, have the smallest clique number and the largest diameter of the 1-skeleton.

\subsection{Cacti}

Here we consider cacti as a simple generalization of trees (see Fig.~\ref{Fig_Cactus}).

\begin{figure}[t]
	\centering
	\begin{tikzpicture}[scale=1.0]
		\node[circle,draw,inner sep=3pt,fill=blue!35] (00) at (0,0) {};
		\node[circle,draw,inner sep=3pt,fill=blue!35] (-11) at (-1,1) {};
		\node[circle,draw,inner sep=3pt,fill=blue!35] (11) at (1,1) {};
		\node[circle,draw,inner sep=3pt,fill=blue!35] (-12) at (-1,2) {};
		\node[circle,draw,inner sep=3pt,fill=blue!35] (12) at (1,2) {};
		\node[circle,draw,inner sep=3pt,fill=blue!35] (-32) at (-3,2) {};
		\node[circle,draw,inner sep=3pt,fill=blue!35] (-22) at (-2,2) {};
		\node[circle,draw,inner sep=3pt,fill=blue!35] (22) at (2,2) {};
		\node[circle,draw,inner sep=3pt,fill=blue!35] (23) at (2,3) {};
		\node[circle,draw,inner sep=3pt,fill=blue!35] (-23) at (-2,3) {};
		\node[circle,draw,inner sep=3pt,fill=blue!35] (03) at (0,3) {};
		\node[circle,draw,inner sep=3pt,fill=blue!35] (13) at (1,3) {};
		\node[circle,draw,inner sep=3pt,fill=blue!35] (33) at (3,3) {};
		
		\draw (00) -- (11) -- (12) -- (03) -- (-12) -- (-11) -- (00);
		\draw (-11) -- (-22) -- (-23) -- (-32) -- (-11);
		\draw (11) -- (22);
		\draw (22) -- (23) -- (33) -- (22);
		\draw (12) -- (13);
		
	\end{tikzpicture}
	\caption {Example of a cactus graph}
	\label {Fig_Cactus}
\end{figure}
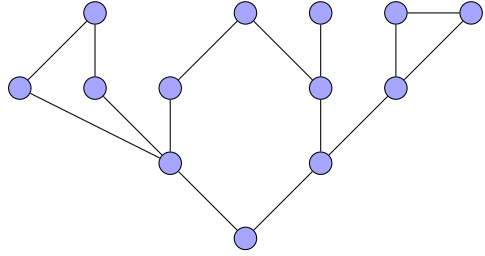

\begin{lemma} \label{Lemma_cactus_adjacency}
	Let $G = (V,E)$ be a cactus. If for two cuts $X, Y \subseteq V$ the vertices $\mathbf{v}(X)$ and $\mathbf{v}(Y)$ are adjacent in $\mathrm{CUT}(G)$, then $|\delta(X \triangle Y)| \leq 2$.
\end{lemma}

\begin{proof}
	Note that if $G = (V,E)$ is a cactus, then for any cut $S \subseteq  V$, where $|\delta(S)| \geq 3$, the graph $H_{S} = (V,E \backslash \delta(S))$ has at least three connected components.
	Indeed, each cactus edge is either a \textit{bridge} (an edge whose removal would disconnect the graph) or belongs to a unique cycle.
	Since each edge belongs to at most one simple cycle, any cut-set can contain either an even number of edges from some cycles or one or more bridges.
	However, removing any bridge, or any two edges from the cycle, breaks the cactus into different connected components.
	Thus, by Theorem~\ref{Theorem_Barahona_adjacency}, if the vertices $\mathbf{v}(X)$ and $\mathbf{v}(Y)$ are adjacent, then $\delta(X \triangle Y)$ contains at most two edges.
\end{proof}

\begin{theorem} \label{Theorem_cactus_diameter}
	Let $G=(V,E)$ be a cactus, then
	\[\left\lfloor \frac{|V|}{2} \right\rfloor \leq d(\mathrm{CUT}(G)) \leq |V|-1.\]
\end{theorem}

\begin{proof}
	An upper bound on the diameter, common to all connected graphs, is achieved on a tree, which is a special case of a cactus.
	
	For the lower bound, consider the spanning tree of the graph $G = (V,E)$. Since every tree is a bipartite graph, we can always construct a non-empty cut $S \subset V$ such that the cut set $\delta(S)$ contains at least $|V|-1$ edges. Therefore, by Lemma~\ref{Lemma_cactus_adjacency}, the shortest path between $\mathbf{v}(S)$ and the empty cut $\mathbf{v}(\emptyset)$ for the cactus graph has length of at least $\lceil \frac{V-1}{2} \rceil = \lfloor \frac{|V|}{2} \rfloor$.
	The exact value is achieved, for example, on any cycle graph.
	
\end{proof}

\begin{theorem} \label{Theorem_cactus_clique}
	Let $G=(V,E)$ be a cactus, then
	\[\omega(\mathrm{CUT}(G)) \leq 2^{\lceil\log_2\left(|E|+1\right)\rceil} \leq 2|E|.\]
\end{theorem}

\begin{proof}

	We bound the clique number of the 1-skeleton of a cut polytope from above by the \textit{chromatic number}, i.e. the smallest number of colors needed to color the vertices of $G$ so that no two adjacent vertices share the same color:
	\[\omega(\mathrm{CUT}(G)) \leq \chi (\mathrm{CUT}(G)).\]

	To color the 1-skeleton of a cut polytope for a cactus, we use the idea of binary representation matrices by Wan~\cite{Wan1997} (see also Linial et al.~\cite{Linial1988}).
	The \textit{binary representation matrix} $\mathrm{BRM}(k)$ is a $k \times \lceil \log_2 k \rceil$ matrix, in which row $i$ ($1 \leq i \leq k$) is the 0/1-vector of length $\lceil\log_2 k \rceil$ obtained from the binary expansion of $i$, adding leading zeroes as necessary. 
	For example, the 3th row of the matrix $\mathrm{BRM}(5)$ has the form $(0,1,1)$ since $3 = 0\cdot 2^{2} + 1 \cdot 2^{1} + 1 \cdot 2^{0}$.
	To each vertex $\mathbf{v}(S)$ of the cut polytope $\mathrm{CUT}(G)$ we assign the color 
	\[\mathrm{color}(\mathbf{v}(S)) = \mathbf{v}(S) \mathrm{BRM}(|E|+1),\]
	where all arithmetic operations are performed modulo 2.

	Let us show that this is a proper graph coloring.	
	Consider two distinct cuts $X,Y \subseteq V$ and the corresponding incidence vectors $\mathbf{v}(X), \mathbf{v}(Y)$, then
	\begin{equation} \label{Eq_BRM_difference}
		\mathbf{v}(X)\mathrm{BRM}(|E|+1) - \mathbf{v}(Y)\mathrm{BRM}(|E|+1) = (\mathbf{v}(X) - \mathbf{v}(Y))\mathrm{BRM}(|E|+1). 
	\end{equation}
	
	If the extreme points $\mathbf{v}(X)$ and $\mathbf{v}(Y)$ are adjacent, then they differ by no more than two coordinates. Therefore, equation (\ref{Eq_BRM_difference}) is either some row of the matrix $\mathrm{BRM}(|E|+1)$, or the difference of two rows. In both cases (\ref{Eq_BRM_difference}) cannot be a zero row, since all rows of the matrix $\mathrm{BRM}(|E|+1)$ are distinct and non-zero. Thus,
	\[\mathbf{v}(X)\mathrm{BRM}(|E|+1) \neq \mathbf{v}(Y)\mathrm{BRM}(|E|+1),\]
	and no two adjacent vertices share the same color.
	
	It remains to note that each color is uniquely represented by a binary string of length $\lceil \log_2 (|E|+1) \rceil$, from which we obtain the required upper bound
	\[\omega(\mathrm{CUT}(G)) \leq \chi (\mathrm{CUT}(G)) \leq 2^{\lceil\log_2\left(|E|+1\right)\rceil} \leq 2|E|.\]
	
\end{proof}

To estimate how tight the upper bound of the Theorem~\ref{Theorem_cactus_clique}, we consider the cycle graph $C_{n} = (V,E)$.
A cycle is split into two connected components if exactly two edges are removed from it. Then, by Theorem~\ref{Theorem_Barahona_adjacency}, the incidence vectors $\mathbf{v}(X)$ and $\mathbf{v}(Y)$ of the cuts $X,Y \subseteq V$ are adjacent in $\mathrm{CUT}(C_n)$ if and only if $|\delta(X \triangle Y)| = 2$ (or $\Delta(\mathbf{v}(X),\mathbf{v}(Y)) = 2$).

Consider a \textit{Hamming ball} of radius 1 centered around some vector $\mathbf{c} \in \{0,1\}^n$ with an odd number of 1's, i.e. set $\mathbf{v} \in \{0,1\}^n$ such that $\Delta (\mathbf{c},\mathbf{v}) = 1$ (see Fig.~\ref{Fig_Hamming_ball}). 
This ball contains $n$ binary vectors, the Hamming distance between any pair of which equals $2$, therefore the corresponding incidence vectors are pairwise adjacent in $\mathrm{CUT}(C_n)$. 
For example, we can consider the cuts $\{\{v_1\},\{v_1,v_2\},\{v_1,v_2,v_3\},\ldots,\{V\}\}$.

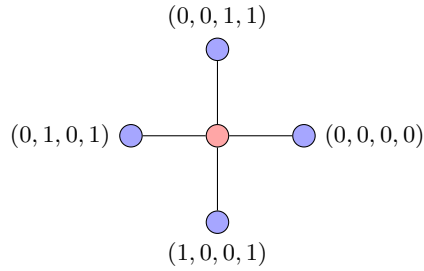
\begin{figure}[t]
	\centering
	\begin{tikzpicture}[scale=1.15]
		\node[circle,draw,inner sep=3pt,fill=blue!35,label=right:{\small $(0,0,0,0)$}] (0000) at (0:1) {};
		\node[circle,draw,inner sep=3pt,fill=blue!35,label=above:{\small $(0,0,1,1)$}] (0011) at (90:1) {};
		\node[circle,draw,inner sep=3pt,fill=blue!35,label=left:{\small $(0,1,0,1)$}] (0101) at (180:1) {};
		\node[circle,draw,inner sep=3pt,fill=blue!35,label=below:{\small $(1,0,0,1)$}] (1001) at (270:1) {};
		\node[circle,draw,inner sep=3pt,fill=red!35] (0001) at (0:0) {};
		
		\draw (0001) -- (0000);
		\draw (0001) -- (0011);
		\draw (0001) -- (0101);
		\draw (0001) -- (1001);
	
	\end{tikzpicture}
	\caption {A Hamming ball of radius $1$ centered around $(0,0,0,1)$ (red vertex)}
	\label {Fig_Hamming_ball}
\end{figure}

Thus, if $G = (V,E)$ is a cactus, then the clique number of the 1-skeleton of $\mathrm{CUT}(G)$ can be equal to $|E|$, while Theorem~\ref{Theorem_cactus_clique} provides an upper bound of $2 |E|$, which is quite tight.
Note also that the 1-skeleton of the cut polytope $\mathrm{CUT}(C_3) = \mathrm{CUT}(K_3)$ (Fig.~\ref{Fig_CUT(K_3)}) is a complete graph whose clique number is $4$.

\subsection{Almost trees (2)}

Now we consider almost trees (2), i.e. such graphs that there are at most $2$ edges not in a spanning tree of each biconnected component
(see Fig.~\ref{Fig_Almost_Tree}).

\begin{figure}[b]
	\centering
	\begin{tikzpicture}[scale=1.0]
		\node[circle,draw,inner sep=3pt,fill=blue!35] (00) at (0,0) {};
		\node[circle,draw,inner sep=3pt,fill=blue!35] (20) at (2,0) {};
		\node[circle,draw,inner sep=3pt,fill=blue!35] (30) at (3,0) {};
		\node[circle,draw,inner sep=3pt,fill=blue!35] (60) at (6,0) {};
		
		\node[circle,draw,inner sep=3pt,fill=blue!35] (11) at (1,1) {};
		\node[circle,draw,inner sep=3pt,fill=blue!35] (31) at (3,1) {};
		\node[circle,draw,inner sep=3pt,fill=blue!35] (51) at (5,1) {};
		\node[circle,draw,inner sep=3pt,fill=blue!35] (71) at (7,1) {};
		
		\node[circle,draw,inner sep=3pt,fill=blue!35] (02) at (0,2) {};
		\node[circle,draw,inner sep=3pt,fill=blue!35] (22) at (2,2) {};
		\node[circle,draw,inner sep=3pt,fill=blue!35] (42) at (4,2) {};
		\node[circle,draw,inner sep=3pt,fill=blue!35] (62) at (6,2) {};
		
		\node[circle,draw,inner sep=3pt,fill=blue!35] (82) at (8,2) {};
		
		\draw (00) -- (20) -- (11) -- (02) -- (00);
		\draw (00) -- (11);
		\draw (11) -- (31);
		\draw (31) -- (22) -- (42) -- (31);
		\draw (31) -- (30);
		\draw (31) -- (51);
		\draw (51) -- (60) -- (71) -- (82) -- (62) -- (51);
		\draw (62) -- (71);

	\end{tikzpicture}
	\caption {Example of an almost tree (2)}
	\label {Fig_Almost_Tree}
\end{figure}
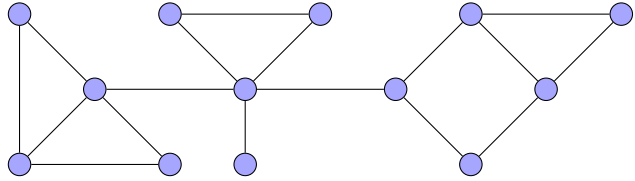

\begin{lemma} \label{Lemma_almost_tree_adjacency}
	Let $G = (V,E)$ be an almost tree (2). If for two cuts $X, Y \subseteq V$ the vertices $\mathbf{v}(X)$ and $\mathbf{v}(Y)$ are adjacent in $\mathrm{CUT}(G)$, then $|\delta(X \triangle Y)| \leq 3$.
\end{lemma}

\begin{proof}
	Similar to Lemma~\ref{Lemma_cactus_adjacency}, if $G = (V,E)$ is an almost tree (2), then for any cut $S \subseteq  V$, where $|\delta(S)| \geq 4$, the graph $H_{S} = (V,E \backslash \delta(S))$ has at least three connected components.
	
	Suppose to the contrary that the cut $S$ splits $G$ into exactly two connected components. Let us consider the spanning trees $T_1 = (V_1,E_1)$ and $T_2 = (V_1,E_2)$ of these components. Let's add to them edges from the cut-set $\delta(S)$. Then the resulting subgraph $G' = (V,E_1 \cup E_2 \cup \delta(S))$ of the original graph $G$ contains a biconnected component whose $3$ edges do not belong to the spanning tree (see Fig~\ref{Fig_Almost_Tree_Cut_4}). Thus, we have a contradiction.

	\begin{figure}[t]
		\centering
		\begin{tikzpicture}[scale=1.0]
			\node[circle,draw,inner sep=3pt,fill=red!35] (A1) at (0.5,2) {};
			\node[circle,draw,inner sep=3pt,fill=red!35] (A2) at (1.5,2) {};
			\node[circle,draw,inner sep=3pt,fill=red!35] (A3) at (2.5,2) {};

			\node[circle,draw,inner sep=3pt,fill=blue!35] (B1) at (0,0) {};
			\node[circle,draw,inner sep=3pt,fill=blue!35] (B2) at (1,0) {};
			\node[circle,draw,inner sep=3pt,fill=blue!35] (B3) at (2,0) {};
			\node[circle,draw,inner sep=3pt,fill=blue!35] (B4) at (3,0) {};
			
			\draw (A1) -- (A2) -- (A3);
			\draw [red] (1.5,2) ellipse (2 and 0.75);
			\node [red] at (3,2) {$T_1$};

			\draw (B1) -- (B2) -- (B3) -- (B4);
			
			\draw [thick,red] (A1) -- (B1);
			\draw [thick,red] (A2) -- (B2);
			\draw [thick,red] (A2) -- (B3);
			\draw [thick,red] (A3) -- (B4);
			
			\draw [blue,dashed] (1.5,0) ellipse (2.5 and 0.75);
			\node [blue] at (3.5,0) {$T_2$};

		\end{tikzpicture}
		\caption {The case with 4 edges in a cut-set separating two connected components}
		\label {Fig_Almost_Tree_Cut_4}
	\end{figure}
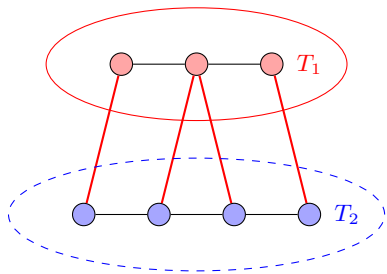
\end{proof}

\begin{theorem}
	Let $G=(V,E)$ be an almost tree (2), then
	\[\left\lfloor \frac{|V|}{3} \right\rfloor \leq d(\mathrm{CUT}(G)) \leq |V|-1.\]
\end{theorem}

\begin{proof}
	The reasoning is completely similar to the proof of Theorem~\ref{Theorem_cactus_diameter}.
	
\end{proof}

\begin{theorem} \label{Theorem_almost_tree_clique}
	Let $G=(V,E)$ be an almost tree (2), then
	\[\omega(\mathrm{CUT}(G)) \leq 2^{\lceil\log_2\left(|E|\right)\rceil + 1} \leq 4 |E|.\]
\end{theorem}

\begin{proof}
	
	We again bound the clique number of the 1-skeleton of a cut polytope by the chromatic number. This time we are adapting the idea of hypercube coloring by Kim et al.~\cite{Kim2000}.
	
	Let $m = \lceil \log_2 k \rceil + 1$. The \textit{modified binary representation matrix} $\mathrm{BRM^{*}}(k)$ is a $k \times m$ matrix, in which row $i$ ($1 \leq i \leq k$) is the $0/1$-vector of the form $(b_i, n_i)$, where $b_i$ is the $(m-1)$-dimensional binary expansion of integer $i-1$, adding leading zeroes as necessary, and
	\[n_i = \begin{cases}
		1, &\text{if $b_i$ contains an even number of 1's};\\
		0, &\text{otherwise}.
	\end{cases}\]
	For example,
	\[\mathrm{BRM^{*}}(5) = \begin{pmatrix}
				0 & 0 & 0 & 1 \\
				0 & 0 & 1 & 0 \\
				0 & 1 & 0 & 0 \\
				0 & 1 & 1 & 1 \\
				1 & 0 & 0 & 0
			\end{pmatrix}.\]
			
	To each vertex $\mathbf{v}(S)$ of the cut polytope $\mathrm{CUT}(G)$ we assign the color 
	\[\mathrm{color}(\mathbf{v}(S)) = \mathbf{v}(S) \mathrm{BRM^{*}}(|E|),\]
	where all arithmetic operations are performed modulo 2. Again, note that we use $2^{m}$ colors in total.

	Consider two cuts $X,Y \subseteq V$ and the incidence vectors $\mathbf{v}(X), \mathbf{v}(Y)$, then
	\begin{equation} \label{Eq_BRM*_difference}
		\mathbf{v}(X)\mathrm{BRM^{*}}(|E|) - \mathbf{v}(Y)\mathrm{BRM^{*}}(|E|) = (\mathbf{v}(X) - \mathbf{v}(Y))\mathrm{BRM^{*}}(|E|). 
	\end{equation}
	
	If the extreme points $\mathbf{v}(X)$ and $\mathbf{v}(Y)$ are adjacent, then they differ by no more than 3 coordinates. Therefore, equation (\ref{Eq_BRM*_difference}) is the modulo 2 sum of at most 3 rows of the matrix $\mathrm{BRM^{*}}(|E|)$, that cannot be a zero row, since no two rows of $\mathrm{BRM^{*}}(|E|)$ are equal, and any odd number of rows of $\mathrm{BRM^{*}}(|E|)$ contain an odd number of 1's in total. Therefore, no two adjacent vertices share the same color. Thus, we obtain the required upper bound
	\[\omega(\mathrm{CUT}(G)) \leq \chi (\mathrm{CUT}(G)) \leq 2^{\lceil\log_2\left(|E|\right)\rceil + 1} \leq 4 |E|.\]

\end{proof}

As for the tightness of this upper estimate, the cut polytope $\mathrm{CUT}(C_n)$ of the cycle graph $C_n = (V,E)$ still has a clique of size $|E|$, while Theorem~\ref{Theorem_almost_tree_clique} provides an upper bound of $4 |E|$. This is not as good as the estimate of Theorem~\ref{Theorem_cactus_clique}, but it still has the same asymptotic behavior.

\section{Complete bipartite and $k$-partite graphs}

In this section, we consider cut polytopes for graphs with a large number of edges, in particular for complete bipartite and $k$-partite graphs. 
A \textit{$k$-partite graph} is a graph whose vertices can be partitioned into $k$ different independent sets $V_1$,$V_2$,\ldots,$V_k$. When $k = 2$ these are the \textit{bipartite graphs}, and when $k = 3$ they are called the \textit{tripartite graphs}.

The situation with the max-cut problem on a bipartite graph is a bit tricky. On the one hand, a \textit{simple max-cut} (all edges have the same weight, for example, equal to one) and a max-cut with non-negative edges have an obvious solution: the largest cut will be the one that separates parts $V_1$ and $V_2$ since the cut-set $\delta(V_1) = \delta(V_2)$ contains all the edges of the graph.
On the other hand, if the edges of the graph can take both positive and negative values, then the max-cut problem is NP-hard (McCormick et al.~\cite{McCormick2003}).
As for $k$-partite graphs, even a simple max-cut is NP-hard for tripartite graphs (Bodlaender and Jansen~\cite{Bodlaender1994}).

\subsection{Complete bipartite graphs}

The diameter of the 1-skeleton of a cut polytope of a complete bipartite graph was first considered by Neto~\cite{Neto2016}.

\begin{theorem}[Neto~\cite{Neto2016}]
	Let $K_{n_1,n_2} = (V_1, V_2, E)$ be a complete bipartite graph, where $n_1 = |V_1| \geq 2$ and $n_2 = |V_2| \geq 2$, then the diameter of the cut polytope $d(\mathrm{CUT}(K_{n_1,n_2})) = 2$.
\end{theorem}

Here we establish that the clique number of the 1-skeleton of a cut polytope of a complete bipartite graph is superpolynomial in dimension.

\begin{theorem} \label{Theorem_bipartite_clique}
	Let $K_{n_1,n_2} = (V_1, V_2, E)$ be a complete bipartite graph, where $n_1 = |V_1| \geq 2$ and $n_2 = |V_2| \geq 2$, then the clique number
	\[\omega(\mathrm{CUT}(K_{n_1,n_2})) \geq 2^{\min\{n_1,n_2\}-1}.\]
\end{theorem}

\begin{proof}
	Let, without loss of generality, $n_1 = \min\{n_1,n_2\} \geq 2$.
	
	Consider the family $\mathcal {F}$ of all possible subsets $S \subseteq V_1$ such that from each pair $S$ and $V_1 \backslash S$ the family $\mathcal {F}$ contains only one subset.
	Then, clearly, $|\mathcal {F}| = 2^{n_1 - 1}$.
	
	Now for each $S \in \mathcal {F}$ we construct a symmetric cut $(S \subseteq V_1) \cup (S \subseteq V_2)$.
	Note that such a cut always exists, since $|V_1| \leq |V_2|$. 
	
	Let $X$ and $Y$ be two symmetric cuts of this type, then in each of the two parts $V_1$ and $V_2$ there is at least one vertex that belongs to $X \triangle Y$, and at least one vertex that does not belong to $X \triangle Y$,
	since no two subsets in $\mathcal {F}$ are the same, and we previously excluded set complements.
	
	However, $K_{n_1,n_2}$ is a complete bipartite graph, and any two vertices from different parts are connected by an edge. Therefore, both $X \triangle Y$ and its complement $V \backslash (X \triangle Y)$ form connected components (see Fig.~\ref{Fig_cut_bipartite}).
	Thus, by Theorem~\ref{Theorem_Barahona_adjacency}, for any two symmetric cuts $X, Y$ of this type (any two subsets of the family $\mathcal{F}$), the vertices $\mathbf{v}(X)$ and $\mathbf{v}(Y)$ of the cut polytope $\mathrm{CUT}(K_{n_1,n_2})$ are pairwise adjacent, which gives us a clique of the required size.
	
\end{proof}

	\begin{figure}[t]
	\centering
	\begin{tikzpicture}[scale=0.95]
		\node[circle,draw,inner sep=3pt,fill=red!35] (A1) at (0,0) {};
		\node[circle,draw,inner sep=3pt,fill=red!35] (A2) at (0,1) {};
		\node[circle,draw,inner sep=3pt,fill=red!35] (A3) at (0,2) {};
		\node[circle,draw,inner sep=3pt,fill=blue!35] (A4) at (0,3) {};
		
		\node[circle,draw,inner sep=3pt,fill=red!35] (B1) at (2,0) {};
		\node[circle,draw,inner sep=3pt,fill=red!35] (B2) at (2,1) {};
		\node[circle,draw,inner sep=3pt,fill=red!35] (B3) at (2,2) {};
		\node[circle,draw,inner sep=3pt,fill=blue!35] (B4) at (2,3) {};
		
		\draw (A1) -- (B1);
		\draw (A1) -- (B2);
		\draw (A1) -- (B3);
		\draw [thick,red] (A1) -- (B4);
		
		\draw (A2) -- (B1);
		\draw (A2) -- (B2);
		\draw (A2) -- (B3);
		\draw [thick,red] (A2) -- (B4);
		
		\draw (A3) -- (B1);
		\draw (A3) -- (B2);
		\draw (A3) -- (B3);
		\draw [thick,red] (A3) -- (B4);
		
		\draw [thick,red] (A4) -- (B1);
		\draw [thick,red] (A4) -- (B2);
		\draw [thick,red] (A4) -- (B3);
		\draw (A4) -- (B4);
		
		\node at (1,-0.75) {cut $X$};

		\begin{scope}[xshift=3.5cm]
			\node[circle,draw,inner sep=3pt,fill=blue!35] (A1) at (0,0) {};
			\node[circle,draw,inner sep=3pt,fill=red!35] (A2) at (0,1) {};
			\node[circle,draw,inner sep=3pt,fill=red!35] (A3) at (0,2) {};
			\node[circle,draw,inner sep=3pt,fill=red!35] (A4) at (0,3) {};
			
			\node[circle,draw,inner sep=3pt,fill=blue!35] (B1) at (2,0) {};
			\node[circle,draw,inner sep=3pt,fill=red!35] (B2) at (2,1) {};
			\node[circle,draw,inner sep=3pt,fill=red!35] (B3) at (2,2) {};
			\node[circle,draw,inner sep=3pt,fill=red!35] (B4) at (2,3) {};
			
			\draw (A1) -- (B1);
			\draw [thick,red] (A1) -- (B2);
			\draw [thick,red] (A1) -- (B3);
			\draw [thick,red] (A1) -- (B4);
			
			\draw [thick,red] (A2) -- (B1);
			\draw (A2) -- (B2);
			\draw (A2) -- (B3);
			\draw (A2) -- (B4);
			
			\draw [thick,red] (A3) -- (B1);
			\draw (A3) -- (B2);
			\draw (A3) -- (B3);
			\draw (A3) -- (B4);
			
			\draw [thick,red] (A4) -- (B1);
			\draw (A4) -- (B2);
			\draw (A4) -- (B3);
			\draw (A4) -- (B4);
			
			\node at (1,-0.75) {cut $Y$};
		\end{scope}

		\begin{scope}[xshift=7cm]
			\node[circle,draw,inner sep=3pt,fill=red!35] (A1) at (0,0) {};
			\node[circle,draw,inner sep=3pt,fill=blue!35] (A2) at (0,1) {};
			\node[circle,draw,inner sep=3pt,fill=blue!35] (A3) at (0,2) {};
			\node[circle,draw,inner sep=3pt,fill=red!35] (A4) at (0,3) {};
			
			\node[circle,draw,inner sep=3pt,fill=red!35] (B1) at (2,0) {};
			\node[circle,draw,inner sep=3pt,fill=blue!35] (B2) at (2,1) {};
			\node[circle,draw,inner sep=3pt,fill=blue!35] (B3) at (2,2) {};
			\node[circle,draw,inner sep=3pt,fill=red!35] (B4) at (2,3) {};
			
			\draw (A1) -- (B1);
			\draw [thick,red] (A1) -- (B2);
			\draw [thick,red] (A1) -- (B3);
			\draw (A1) -- (B4);
			
			\draw [thick,red] (A2) -- (B1);
			\draw (A2) -- (B2);
			\draw (A2) -- (B3);
			\draw [thick,red] (A2) -- (B4);
			
			\draw [thick,red] (A3) -- (B1);
			\draw (A3) -- (B2);
			\draw (A3) -- (B3);
			\draw [thick,red] (A3) -- (B4);
			
			\draw (A4) -- (B1);
			\draw [thick,red] (A4) -- (B2);
			\draw [thick,red] (A4) -- (B3);
			\draw (A4) -- (B4);
			
			\node at (1,-0.75) {cut $X \triangle Y$};
		\end{scope}

		\begin{scope}[xshift=10.5cm]
			\node[circle,draw,inner sep=3pt,fill=red!35] (A1) at (0,0) {};
			\node[circle,draw,inner sep=3pt,fill=blue!35] (A2) at (0,1) {};
			\node[circle,draw,inner sep=3pt,fill=blue!35] (A3) at (0,2) {};
			\node[circle,draw,inner sep=3pt,fill=red!35] (A4) at (0,3) {};
			
			\node[circle,draw,inner sep=3pt,fill=red!35] (B1) at (2,0) {};
			\node[circle,draw,inner sep=3pt,fill=blue!35] (B2) at (2,1) {};
			\node[circle,draw,inner sep=3pt,fill=blue!35] (B3) at (2,2) {};
			\node[circle,draw,inner sep=3pt,fill=red!35] (B4) at (2,3) {};
			
			\draw (A1) -- (B1);
			\draw (A1) -- (B4);
			
			\draw (A2) -- (B2);
			\draw (A2) -- (B3);
			
			\draw (A3) -- (B2);
			\draw (A3) -- (B3);
			
			\draw (A4) -- (B1);
			\draw (A4) -- (B4);
			
			\node at (1,-0.75) {graph $H_{X \triangle Y}$};
		\end{scope}
		
	\end{tikzpicture}
	\caption {Symmetric cuts of a complete bipartite graph $K_{4,4}$}
	\label {Fig_cut_bipartite}
\end{figure}
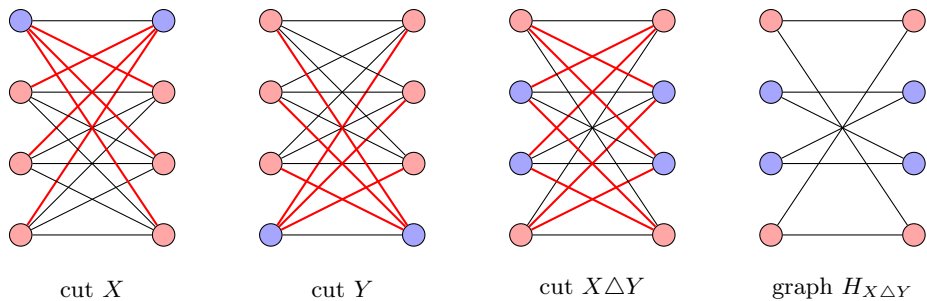

\subsection{Complete tripartite and $k$-partite graphs}

\begin{theorem}
	Let $K_{n_1,n_2,n_3} = (V_1, V_2, V_3, E)$ be a complete tripartite graph with the node set $V = V_1 \cup V_2 \cup V_3$, where $n_i = |V_i| \geq 2$, $1 \leq i \leq 3$, then the diameter of the cut polytope $d(\mathrm{CUT}(K_{n_1,n_2,n_3})) = 2$.
\end{theorem}

\begin{proof}
	We consider two cuts $X,Y \subseteq V$ and explore three possible cases. 
	
	\textit{Case 1.}
	If $V_i \cap (X \triangle Y) \neq \emptyset$ and $V_i \backslash (X \triangle Y) \neq \emptyset$ for at least two of the three parts $V_i$ (possibly different), then the extreme points $\mathbf{v}(X)$ and $\mathbf{v}(Y)$ are adjacent in the cut polytope $\mathrm{CUT}(K_{n_1,n_2,n_3})$. Indeed, both $X \triangle Y$ and its complement $V \backslash (X \triangle Y)$ form connected components (see Fig.~\ref{Fig_cut_bipartite}).
	
	\textit{Case 2.}
	Let for some two parts, without loss of generality $V_1$ and $V_2$, $V_1 \cap (X \triangle Y) = V_2 \cap (X \triangle Y) = \emptyset$. 
	Therefore, $X \triangle Y \subseteq V_3$, and if $|X \triangle Y| \geq 2$, then the symmetric difference does not form a connected component, and the vectors $\mathbf{v}(X)$ and $\mathbf{v}(Y)$ are not adjacent.
	
	We construct a cut $Z \subseteq V$ such that $|V_1 \cap (Z \triangle X)| = 1$ and $|V_2 \cap (Z \triangle X)| = 1$ (see Fig.~\ref{Fig_tripartite_diameter_empty_difference}). This is always possible since both $V_1$ and $V_2$ contain at least two vertices. It is enough to add or subtract a vertex from $X$ in each of the parts. We obtain that $Z \triangle X$ has exactly one vertex each in $V_1$ and $V_2$. Then, $Z \triangle X$ and $V \backslash (Z \triangle X)$ form connected components. The same is true for $Z \triangle Y$. Therefore, the vector $\mathbf{v}(Z)$ is adjacent to both $\mathbf{v}(X)$ and $\mathbf{v}(Y)$.
	
	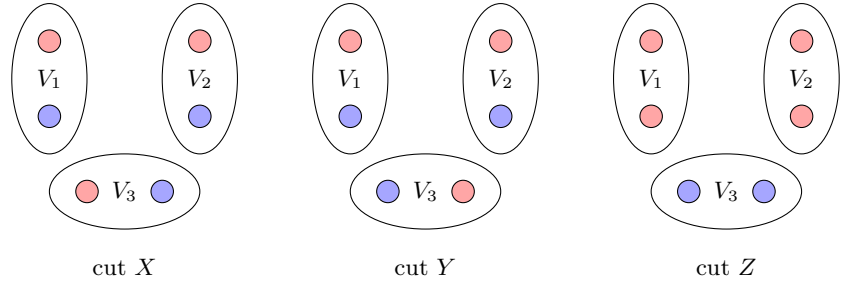
\begin{figure}[t]
		\centering
		\begin{tikzpicture}[scale=1.0]
			\node[circle,draw,inner sep=3pt,fill=blue!35] (A1) at (0,0) {};
			\node[circle,draw,inner sep=3pt,fill=red!35] (A2) at (0,1) {};
			\draw (0,0.5) ellipse (0.5 and 1);
			\node at (0,0.5) {$V_1$};
			
			\node[circle,draw,inner sep=3pt,fill=blue!35] (B1) at (2,0) {};
			\node[circle,draw,inner sep=3pt,fill=red!35] (B2) at (2,1) {};
			\draw (2,0.5) ellipse (0.5 and 1);
			\node at (2,0.5) {$V_2$};
			
			\node[circle,draw,inner sep=3pt,fill=red!35] (C1) at (0.5,-1) {};
			\node[circle,draw,inner sep=3pt,fill=blue!35] (C2) at (1.5,-1) {};
			\draw (1,-1) ellipse (1 and 0.5);
			\node at (1,-1) {$V_3$};

			\node at (1,-2) {cut $X$};

			\begin{scope}[xshift=4cm]
				\node[circle,draw,inner sep=3pt,fill=blue!35] (A1) at (0,0) {};
				\node[circle,draw,inner sep=3pt,fill=red!35] (A2) at (0,1) {};
				\draw (0,0.5) ellipse (0.5 and 1);
				\node at (0,0.5) {$V_1$};
				
				\node[circle,draw,inner sep=3pt,fill=blue!35] (B1) at (2,0) {};
				\node[circle,draw,inner sep=3pt,fill=red!35] (B2) at (2,1) {};
				\draw (2,0.5) ellipse (0.5 and 1);
				\node at (2,0.5) {$V_2$};
				
				\node[circle,draw,inner sep=3pt,fill=blue!35] (C1) at (0.5,-1) {};
				\node[circle,draw,inner sep=3pt,fill=red!35] (C2) at (1.5,-1) {};
				\draw (1,-1) ellipse (1 and 0.5);
				\node at (1,-1) {$V_3$};
				
				\node at (1,-2) {cut $Y$};
			\end{scope}

			\begin{scope}[xshift=8cm]
				\node[circle,draw,inner sep=3pt,fill=red!35] (A1) at (0,0) {};
				\node[circle,draw,inner sep=3pt,fill=red!35] (A2) at (0,1) {};
				\draw (0,0.5) ellipse (0.5 and 1);
				\node at (0,0.5) {$V_1$};
				
				\node[circle,draw,inner sep=3pt,fill=red!35] (B1) at (2,0) {};
				\node[circle,draw,inner sep=3pt,fill=red!35] (B2) at (2,1) {};
				\draw (2,0.5) ellipse (0.5 and 1);
				\node at (2,0.5) {$V_2$};
				
				\node[circle,draw,inner sep=3pt,fill=blue!35] (C1) at (0.5,-1) {};
				\node[circle,draw,inner sep=3pt,fill=blue!35] (C2) at (1.5,-1) {};
				\draw (1,-1) ellipse (1 and 0.5);
				\node at (1,-1) {$V_3$};
				
				\node at (1,-2) {cut $Z$};
			\end{scope}
			
		\end{tikzpicture}
		\caption {Construction of the cut $Z$ in the case of $V_1 \cap (X \triangle Y) = V_2 \cap (X \triangle Y) = \emptyset$}
		\label {Fig_tripartite_diameter_empty_difference}
	\end{figure}

	\textit{Case 3.}
	Now let that for some two parts, without loss of generality $V_1$ and $V_2$, $V_1 \backslash (X \triangle Y) = \emptyset$ and $V_2 \backslash (X \triangle Y) = \emptyset$.
	Then both cuts $X$ and $Y$ cannot simultaneously contain all vertices from $V_1$ and $V_2$. Without loss of generality, we assume that $X \cap V_1 \neq V_1$ and $Y \cap V_2 \neq V_2$. 
	
	Let us choose vertices $v_1 \in V_1 \backslash X$ and $v_2 \in V_2 \backslash Y$ and construct a cut $Z \subseteq V$ (see Fig.~\ref{Fig_tripartite_diameter_complete_difference}) such that
	\[Z \cap V_1 = (X \cap V_1) \cup \{v_1\} \text{ and }  Z \cap V_2 = (Y \cap V_2) \cup \{v_2\}.\]
	It only remains to note that
	\begin{itemize}
		\item $V_1 \cap (X \triangle Z) = \{v_1\}$ and $V_2 \cap (X \triangle Z) = V_2 \backslash \{v_2\}$, hence the vectors $\mathbf{v}(X)$ and $\mathbf{v}(Z)$ are adjacent;
		\item $V_1 \cap (Y \triangle Z) = V_1 \backslash \{v_1\}$ and $V_2 \cap (Y \triangle Z) = \{v_2\}$, hence the vectors $\mathbf{v}(Y)$ and $\mathbf{v}(Z)$ are adjacent.
	\end{itemize} 
	
	\begin{figure}[t]
		\centering
		\begin{tikzpicture}[scale=1.0]
			\node[circle,draw,inner sep=3pt,fill=blue!35] (A1) at (0,0) {};
			\node[circle,draw,inner sep=3pt,fill=blue!35] (A2) at (0,1) {};
			\draw (0,0.5) ellipse (0.5 and 1);
			\node at (0,0.5) {$V_1$};
			
			\node[circle,draw,inner sep=3pt,fill=red!35] (B1) at (2,0) {};
			\node[circle,draw,inner sep=3pt,fill=red!35] (B2) at (2,1) {};
			\draw (2,0.5) ellipse (0.5 and 1);
			\node at (2,0.5) {$V_2$};
			
			\node[circle,draw,inner sep=3pt,fill=blue!35] (C1) at (0.5,-1) {};
			\node[circle,draw,inner sep=3pt,fill=blue!35] (C2) at (1.5,-1) {};
			\draw (1,-1) ellipse (1 and 0.5);
			\node at (1,-1) {$V_3$};
			
			\node at (1,-2) {cut $X$};

			\begin{scope}[xshift=4cm]
				\node[circle,draw,inner sep=3pt,fill=red!35] (A1) at (0,0) {};
				\node[circle,draw,inner sep=3pt,fill=red!35] (A2) at (0,1) {};
				\draw (0,0.5) ellipse (0.5 and 1);
				\node at (0,0.5) {$V_1$};
				
				\node[circle,draw,inner sep=3pt,fill=blue!35] (B1) at (2,0) {};
				\node[circle,draw,inner sep=3pt,fill=blue!35] (B2) at (2,1) {};
				\draw (2,0.5) ellipse (0.5 and 1);
				\node at (2,0.5) {$V_2$};
				
				\node[circle,draw,inner sep=3pt,fill=blue!35] (C1) at (0.5,-1) {};
				\node[circle,draw,inner sep=3pt,fill=blue!35] (C2) at (1.5,-1) {};
				\draw (1,-1) ellipse (1 and 0.5);
				\node at (1,-1) {$V_3$};
				
				\node at (1,-2) {cut $Y$};
			\end{scope}

			\begin{scope}[xshift=8cm]
				\node[circle,draw,inner sep=3pt,fill=red!35] (A1) at (0,0) {};
				\node[circle,draw,inner sep=3pt,fill=blue!35] (A2) at (0,1) {};
				\draw (0,0.5) ellipse (0.5 and 1);
				\node at (0,0.5) {$V_1$};
				
				\node[circle,draw,inner sep=3pt,fill=red!35] (B1) at (2,0) {};
				\node[circle,draw,inner sep=3pt,fill=blue!35] (B2) at (2,1) {};
				\draw (2,0.5) ellipse (0.5 and 1);
				\node at (2,0.5) {$V_2$};
				
				\node[circle,draw,inner sep=3pt,fill=blue!35] (C1) at (0.5,-1) {};
				\node[circle,draw,inner sep=3pt,fill=blue!35] (C2) at (1.5,-1) {};
				\draw (1,-1) ellipse (1 and 0.5);
				\node at (1,-1) {$V_3$};
				
				\node at (1,-2) {cut $Z$};
			\end{scope}
			
		\end{tikzpicture}
		\caption {Construction of the cut $Z$ in the case of $V_1 \backslash (X \triangle Y) = \emptyset$ and $V_2 \backslash (X \triangle Y) = \emptyset$}
		\label {Fig_tripartite_diameter_complete_difference}
	\end{figure}
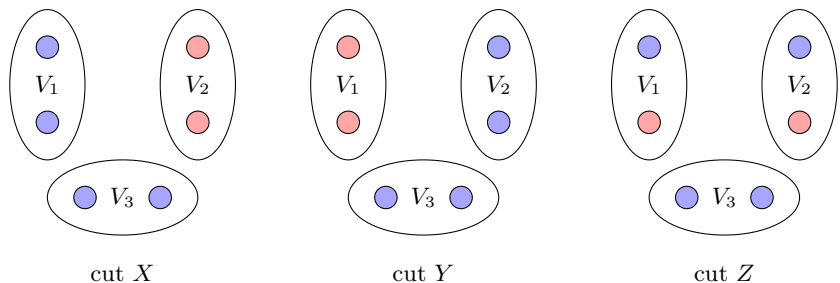
	
	Thus, for any pair of non-adjacent vertices $\mathbf{v}(X)$ and $\mathbf{v}(Y)$ of a cut polytope $\mathrm{CUT}(K_{n_1,n_2,n_3})$ of a complete tripartite graph, we can construct a third vertex $\mathbf{v}(Z)$ adjacent to both. Therefore, $d(\mathrm{CUT}(K_{n_1,n_2,n_3})) = 2$.

	\begin{theorem}
		Let $K_{n_1,\ldots,n_k} = (V_1, \ldots, V_k, E)$ be a complete $k$-partite graph with the node set $V = V_1 \cup \ldots \cup V_k$, where $n_i = |V_i| \geq 2$, $1 \leq i \leq k$, then the diameter of the cut polytope $d(\mathrm{CUT}(K_{n_1,\ldots,n_k})) = 2$.
	\end{theorem}
	
	\begin{proof}
		The cases of $k=2$ and $k=3$ have already been considered above, and for the remaining $k$ the proof is similar to complete tripartite graphs.
		
	\end{proof}

\end{proof}

\begin{theorem}
	Let $K_{n_1,n_2,\ldots,n_k} = (V_1, V_2, \ldots, V_k, E)$ be a complete $k$-partite graph, where $n_1 = |V_1| \geq n_2 = |V_2| \geq \ldots \geq n_k = |V_k|$, then the clique number
	\[\omega(\mathrm{CUT}(K_{n_1,n_2,\ldots,n_k})) \geq 2^{n_2 - 1}.\]
\end{theorem}

\begin{proof}
	The idea is to limit the consideration to only the parts $V_1$ and $V_2$, then repeat the construction from the proof of the Theorem~\ref{Theorem_bipartite_clique}.
	
	We again consider a family $\mathcal {F}$ of all possible subsets $S \subseteq V_2$ such that from each pair $S$ and $V_2 \backslash S$ the family $\mathcal {F}$ contains only one subset.
	For each $S \in \mathcal {F}$ we construct a symmetric cut $(S \subseteq V_1) \cup (S \subseteq V_2)$. 
	For any two symmetric cuts $X, Y$ of this type, both $X \triangle Y$ and $V \backslash (X \triangle Y)$ form connected components, and the vectors $\mathbf{v}(X)$ and $\mathbf{v}(Y)$ of the cut polytope are pairwise adjacent, which gives us a clique of the size $|\mathcal {F}| = 2^{n_2 - 1}$.

\end{proof}

\section{Conclusion}

The results of the research are summarized in Table~\ref{Table_pivot_1-skeleton_cut_polytopes}.
We can see that as the number of edges in the graph $G$ grows, the diameter of the 1-skeleton of the cut polytope $\mathrm{CUT}(G)$ decreases, and the clique number increases.
This is expected, since if $G'$ is a subgraph of $G$, then 1-skeleton of $\mathrm{CUT}(G')$ is a subgraph of the 1-skeleton of $\mathrm{CUT}(G)$.

However, the clique number of the 1-skeleton undergoes a qualitative leap as we move from trees and almost trees to complete bipartite graphs. On the one hand, the max-cut problem is polynomially solvable for trees, cacti, and almost trees (2). The clique number of the 1-skeleton here is linear in the dimension. On the other hand, for complete bipartite and $k$-partite graphs, the max-cut problem is NP-hard, and the clique number of the 1-skeleton of the corresponding cut polytopes is superpolynomial in the dimension.

\begin{credits}
\subsubsection{\ackname} 
The research is supported by the Yaroslavl region grant No. 4-np/2023 and the P.G. Demidov Yaroslavl State University Project VIP-016.


\subsubsection{\discintname}
The author has no competing interests to declare that are relevant to the content of this article.

\end{credits}
%
%
%
 \bibliographystyle{splncs04}
 \bibliography{Nikolaev_arXiv_2024}

\end{document}